\documentclass{amsart}
\usepackage[centertags]{amsmath} \usepackage{amsfonts}
\usepackage{amssymb} \usepackage{amsthm}
\usepackage{newlfont}
 \DeclareMathOperator{\sym}{Sym}
 \DeclareMathOperator{\soc}{soc}
 \DeclareMathOperator{\out}{Out}

 \DeclareMathOperator{\frat}{Frat}

\DeclareMathOperator{\GL}{GL}

\renewcommand{\emptyset}{\varnothing}

\newcommand{\sg}{\sigma}

\newcommand{\Om}{\Omega}
\newcommand{\cm}{\mathcal M}
\newtheorem{thm}{Theorem}%[section]

 \newtheorem{lemma}[thm]{Lemma}
\newtheorem{prop}[thm]{Proposition}

\numberwithin{equation}{section}

\renewcommand{\footnote}{\endnote}
\newcommand{\ignore}[1]{}\makeglossary

\begin{document}
\bibliographystyle{amsplain}
\title{Direct products of finite groups\\ as unions of proper subgroups}
\author{Martino Garonzi}
\address{Dipartimento di Matematica Pura ed Applicata \\
Via Trieste 63, 35121 Padova, Italy.}
\email{mgaronzi@math.unipd.it}
\author{Andrea Lucchini}
\address{Dipartimento di Matematica Pura ed Applicata \\
Via Trieste 63, 35121 Padova, Italy.}
\email{lucchini@math.unipd.it}
\thanks{Research partially supported
by MIUR-Italy via PRIN "Group theory and applications"}


\date{}

\begin{abstract}
We determine all the ways in which a direct product of two finite groups can be expressed as the set-theoretical union of  proper subgroups  in a family of
minimal cardinality.
\end{abstract}

\maketitle

%\renewcommand{\thefootnote}{}
%\footnotetext{\textit{Key words:} $n$-sum groups; minimal coverings; monolithic groups}
%\footnotetext{\textit{2000 Mathematics Subject Classification 20D60}

\section{Introduction}

If $G$ is a non cyclic finite group, then there exists a finite collection of proper subgroups whose set-theoretical union is all of $G;$ such a collection is called a {\sl{cover}} for $G.$ A minimal cover is one of least cardinality and the size of a minimal cover of $G$ is denoted by $\sigma(G)$ (and for convenience we shall write $\sigma(G)=\infty$ if $G$ is cyclic). The study of  minimal covers was introduced
by J.H.E. Cohn \cite{cohn}. If $N$ is a normal subgroup of $G,$
then $\sigma(G)\leq \sigma(G/N);$ indeed a cover of $G/N$ can be lifted to a cover of $G.$
In particular, as it was noticed in \cite{cohn}, if $G=H_1\times H_2$ is the direct product of two finite groups,
then $\sigma(H_1\times H_2)\leq \min \{\sigma(H_1),\sigma(H_2)\}.$ It is easy to prove that if $|H_1|$ and $|H_2|$ are coprime numbers, then $\sigma(H_1\times H_2) = \min \{\sigma(H_1),\sigma(H_2)\}$ (see \cite[Lemma 4]{cohn}). The situation is different if $|H_1|$ and $|H_2|$ are not coprime; for example if $H_1\cong H_2\cong C_p$ are cyclic groups of order $p$ and $p$ is a prime, then $\sigma(H_1\times H_2)=p+1$
since a minimal cover of $H_1\times H_2$ must contain all the $p+1$ maximal subgroups. Some partial results
are contained in \cite{cubo}.
In this paper we obtain a complete and general answer to the question how $\sigma(H_1\times H_2)$ is related with $\sigma(H_1)$ and $\sigma(H_2).$
\begin{thm}
If $G=H_1 \times H_2$ is the direct product of two finite groups, then either $\sg(G)=\min \{ \sg(H_1), \sg(H_2)\}$ or $\sg(G)=p+1$ and the cyclic group of order $p$ is a homomorphic image of both    $H_1$ and  $H_2$.
\end{thm}
This is a consequence of a more general result, describing all the possible minimal covers of a direct product.
\begin{thm}\label{main}
Let $\cal M$ be a minimal cover of  a direct product $G=H_1 \times H_2$ of two finite groups.
Then one of the following holds:
\begin{enumerate}
\item  $\cal M=\{X\times H_2\mid X\in \cal X\}$ where $\cal X$ is a minimal cover of $H_1.$ In this case  $\sigma(G)=\sigma(H_1).$
\item  $\cal M=\{H_1\times X\mid X\in \cal X\}$ where $\cal X$ is a minimal cover of $H_2.$ In this case  $\sigma(G)=\sigma(H_2).$
\item  There exist $N_1\trianglelefteq H_1,$ $N_2\trianglelefteq H_2$ with $H_1/N_1\cong H_2/N_2\cong C_p$
and $\cal M$ consists of the maximal subgroups of $H_1\times H_2$ containing $N_1\times N_2.$ In this case $\sigma(G)=p+1.$
\end{enumerate}
\end{thm}

\section{Proofs of the theorems}
First we recall some elementary results on the minimal covers.

\begin{lemma}\label{effe} If $\cal X$ is a minimal cover of a finite group $Y$ and $F$ is a normal subgroup
of $Y$ such that $FX\neq Y$ for  each $X\in \cal X,$ then $F \leq X$ for each $X\in \cal X.$
\end{lemma}
\begin{proof} Let $\cal X=\{X_1,\dots,X_\sigma\}$.
By our assumption,
$\{X_1F,\dots,X_\sigma F\}$ is also a minimal cover of $Y.$ In particular, for each $i$, there
exists $x_i\in X_iF$ such that $x_i\notin X_jF$ if $j\neq i.$ Assume by contradiction $F\not\leq X_i$ and take $f\in F\setminus X_i:$ we have
$x_i\cdot g\in X_i$ for some $g \in F$ and consequently $x_i\cdot g\cdot f\notin X_i$: this implies $x_i\cdot g\cdot f \in X_j$ for some $j\neq i,$
but then $x_i \in X_jF,$ a contradiction.
\end{proof}

\begin{lemma}\cite[Lemma 3.2]{tom}\label{lemmatom} Let $N$ be a proper normal subgroup of the finite group $G$. Let $U_1,\dots,U_h$ be proper subgroups of $G$ containing $N$ and $V_1\dots,V_k$ be proper subgroups such that $V_iN=G$ with $|G:V_i|=\beta_i$ and $\beta_1\leq \dots\leq  \beta_k.$ If $$G=U_1\cup \dots \cup U_h \cup V_1\cup \dots \cup V_k\text{\quad but \quad }G\neq U_1\cup \dots \cup U_h,$$ then $\beta_1\leq k.$ Furthermore, if $\beta_1=k,$ then $\beta_1=\dots=\beta_k$ and $V_i\cap V_j \leq
U_1\cup \dots \cup U_h$ for all $i\neq j.$
\end{lemma}

The other tool that we need in the proof is a description of the maximal subgroups of a direct product $H_1\times H_2.$
\begin{itemize}
\item
We will say that a maximal subgroup $M$ of $H_1\times H_2$ is of {\sl{standard type}} if either $M=X_1 \times H_2$ with
$X_1$ a maximal subgroup of $H_1$ or $M=H_1 \times X_2$ with
$X_2$ a maximal subgroup of $H_2$.
\item We will say that a maximal subgroup $M$ of $H_1\times H_2$ is of {\sl{diagonal type}} if there exist
a maximal normal subgroup $N_1$ of $H_1$, a maximal normal subgroup $N_2$ of $H_2$ and an isomorphism
$\phi: H_1/N_1 \to H_2/N_2$ such that $M=\{(h_1,h_2)\in H_1\times H_2 \mid \phi(h_1N_1)=h_2N_2\}.$
\end{itemize}
By \cite[Chap. 2, (4.19)]{suzuki}, the following holds.
\begin{lemma}\label{suz} A maximal subgroup of $H_1\times H_2$ is either of standard type or of diagonal type.
\end{lemma}

\begin{lemma}\label{diagon} Let $\cal M=\{M_1,\dots,M_\sigma\}$ be a minimal cover of $G=H_1\times H_2.$
If all the subgroups in $\cal M$ are maximal and $\cal M$ contains a subgroup of diagonal type whose index is a prime number $p$, then
$\sigma(G)=p+1$ and all the subgroups in M are normal of index p.
\end{lemma}

\begin{proof}
First notice that if $\cal M$ contains a maximal subgroup of diagonal type and index $p$, then $C_p\times C_p$ is an epimorphic image
of $G$ and consequently $$\sigma(G)\leq \sigma(C_p\times C_p)=p+1.$$

We argue by induction on the order of $G$. We may assume that there exists no nontrivial normal subgroup
$N$ of $G$ such that $N\leq M$ for all $M\in \cal M$ and $N\leq H_1$. Otherwise $\{M_1/N,\dots,M_\sigma/N\}$
would be a minimal cover of $(H_1/N)\times H_2$ containing a maximal diagonal subgroup of index $p$ and the
conclusion follows by induction. For the same reason, there is no nontrivial normal subgroup
$N$ of $G$ such that $N\leq M$ for all $M\in \cal M$ and $N\leq H_2$. In particular
$$\frat G =\frat H_1 \times \frat H_2=1.$$

%Let $\soc H_1=M_1\times \dots \times M_u$ and $\soc H_2=N_1\times \dots N_v$ with $M_i$ and $N_j$ minimal normal subgroups of $G.$
First assume that $Z(G)$ has order divisible by $p.$ This implies that there exists a central subgroup, say $N,$ of order $p$, which is contained either in $H_1$ or in $H_2$. Let $\cal U$ be the set of subgroups in $\cal M$ not containing $N$. By our assumption $\cal U\neq \emptyset$,
moreover if $M\in \cal U$, then $G=M\times N$ and in particular $M$ is a normal subgroup of $G$ and has index $p$. By Lemma \ref{lemmatom}, $p\leq |\cal U|\leq \sigma(G)\leq p+1.$ Moreover $N$ is not contained in the union of the subgroups in $\cal U$, so we must  have $|\cal U|=p$ and $\sigma(G)=p+1.$ Let $M$ be
unique element of  $\cal M\setminus \cal U.$ By Lemma \ref{lemmatom}, $M$ contains the intersection $M_i\cap M_j$ of any two different subgroups in $\cal U$,
but $G/(M_i\cap M_j)\cong C_p \times C_p,$ so $M$ is a normal subgroup of index $p.$

Now assume that $p$ does not divide $|Z(G)|.$ Write $\soc (G)=N_1\times \dots \times N_t$ as a product of minimal normal subgroups.
We may assume that each $N_i$ is contained either in $H_1$ or in $H_2$ and  that $N_i$ is abelian if and only if $i<u.$ Let $C=\bigcap_{1\leq i\leq t}C_G(N_i).$
Since $\frat G=1,$ the socle of $G$ coincides with the generalized Fitting subgroup of $G$ and, by the Bender $F^*$-Theorem (see for example \cite[(31.13)]{asch}),
$$C=C_G(\soc G )=Z(\soc G )=\prod_{i<u}N_i.$$
Since $p$ does not divide $|Z(G)|$, if $N_i$ is a $p$-group, then $N_i=[N_i,G]\leq G^\prime \cap C.$
In particular $p$ does not divide $|C: G^\prime \cap C|=|CG^\prime : G^\prime|$. On the other hand $p$ divides
$|G:G^\prime|=|G: CG^\prime||CG^\prime : G^\prime|$, hence $G/C$ has $C_p$ as an epimorphic image. Since $G/C$
is a subdirect product of $\prod_{1\leq i \leq t} G/C_G(N_i),$  there must exist a minimal normal subgroup $N$ of $G$ which is contained in either $H_1$ or $H_2$ and with the property that $A=G/C_G(N)$ has a chief factor of order $p.$ By our assumption the set $\cal U$ of the subgroups in $\cal M$ not containing $N$ is non empty. By Lemma \ref{lemmatom}, $p+1\geq \sigma(G)\geq |\cal U|
\geq \beta$, with $\beta=\min_{M\in \cal U}|G:M|.$ Fix a maximal subgroup $M$ in $\cal U$ with $|G:M|=\beta.$

If $N$ is abelian, then the subgroups in $\cal U$ are complements of $N$,
hence $\beta=|N|.$ Moreover $N$ is not contained in the union of the subgroups in $\cal U,$
hence $p+1\geq \sigma(G)\geq |N|+1$. However $p$ must be a prime divisor of $|A|$, but $A\leq \GL(N)$
and this implies $p < |N|$, a contradiction.

If $N$ is a non-abelian simple group, then $C_p$ is isomorphic to a chief factor of a subgroup of $\out(N)$ hence $p \leq |\out(N)|$.
However $\beta=|G:M|=|N:M\cap N|$ is the index of a proper subgroup of $N$ so in particular $\beta>2p$
(see e.g.  \cite[Lemma 2.7]{asch-gur}).
But then $p+1\geq \beta >2p,$ a contradiction.

We are left with the case $N=S_1\times \dots \times S_r \cong S^r$ where $S$ is a  nonabelian simple group. Let $\pi_i: N\to S_i$ the projection to the $i$-th factor of $N.$ Since $MN=G$ and $N$ is a minimal normal subgroup of $G$, the maximal subgroup $M$ permutes transitively the minimal normal subgroups $S_1,\dots,S_r$ of $N$ and normalizes $M\cap N$. This implies that $\pi_1(M\cap N)\cong \dots \cong \pi_r(M\cap N)$
so either $M\cap N \leq T_1 \times \dots \times T_r$ with $T_i < S_i$ for each $i \in \{1,\dots,r\}$ or (see for example \cite[Proposition 1.1.44]{classes}) $M\cap N\cong S^u$ with $u$ a proper divisor of $r.$
Therefore, by \cite[Lemma 2.7]{asch-gur}, $$p+1 \geq \beta=|N:M\cap N|\geq \min \{ 2^rq^r,|S|^{r/2}\}$$
with $q$ the largest prime divisor of $|\out(S)|.$
Moreover $C_p$ is isomorphic to a chief factor of a subgroup of $\out(N)\cong \out(S) \wr \sym (r)$, so either $p$ divides $|\sym (r)|$,
in which case $p\leq r,$ or $p$ divides $|\out S|$ and consequently $p\leq q.$ Both these cases lead to a contradiction.
\end{proof}

\begin{prop} Let $\cal X=\{X_1,\dots,X_\sigma\}$ be a minimal cover of $G=H_1\times H_2.$
If a subgroup of $\cal X$ is contained in a maximal subgroup of diagonal type whose index is a prime number $p$, then
$\sigma(G)=p+1$, $X_i$ is a normal subgroup of index $p$ for each $i \in \{1,\dots,\sigma\}$ and
$\bigcap_i X_i$ has index $p^2$ in $G.$
\end{prop}
\begin{proof} For each $i\in \{1,\dots,\sigma\},$ let $M_i$ be a maximal subgroup of $G$ containing $X_i$, chosen is such a way that
$M_i$ is a maximal subgroup of diagonal type and index $p$ when $X_i$ is contained in such a maximal subgroup.
The cover $\cal M=\{M_1,\dots,M_\sigma\}$ satisfies the hypothesis of Lemma \ref{diagon}, so $\sigma=p+1$ and $M_i$ is a maximal normal subgroup of index $p$ for each
$i\in \{1,\dots,\sigma\}.$ Let $N=M_1\cap M_2$. If, by contradiction, there exists $i \in \{3,\dots,\sigma\}$ such that $M_i$ does not contain $N,$ then, by Lemma \ref{lemmatom}, $\sigma \geq 2+p.$ So for each $i\in \{1,\dots,\sigma\}$, we have $N\leq M_i$ but then $X_iN\leq M_i\neq G$, hence $N \leq X_i$ by Lemma \ref{effe}.
In particular $\{X_1/N,\dots,X_\sigma/N\}$ is a cover of $G/N\cong C_p\times C_p.$ Since $\sigma(C_p\times C_p)=p+1=\sigma$, $X_i\neq N$ for each $i \in \{1,\dots,\sigma\}.$
\end{proof}

\begin{prop} Let $\cal X=\{X_1,\dots,X_\sigma\}$ be a minimal cover of $G=H_1\times H_2.$
If $\cal X$ contains no subgroup of diagonal type whose index is a prime number, then either $H_1\times 1$ or $1\times H_2$
is contained in $\bigcap_{1\leq i\leq \sigma} X_i.$
%for each $i \in\  \{1,\dots,\sigma\}.$
\end{prop}

\begin{proof}
For each $i\in \{1,\dots,\sigma\},$ let $M_i$ be a maximal subgroup of $G$ containing $X_i$.
We have that  $\cal M=\{M_1,\dots,M_\sigma\}$ is a minimal cover of  $G$  given by $\sigma=\sg(G)$ maximal subgroups of $G$.
   We set:
   \begin{eqnarray*}
   \cal M_1&=&\{M \in \cm \mid M \geq \ H_2\}=\{L \times H_2 \mid L  \textrm{ a maximal subgroup of } \ H_1\},\\
   \cal M_2&=&\{M \in \cm \mid M \geq \ H_1\}=\{H_1 \times L \mid L  \textrm{ a maximal subgroup of } \ H_2\},\\
   \cal M_3&=& \cal M \setminus (\cal M_1 \cup \cal M_2 ).
  \end{eqnarray*}
Then we define the two sets
$$\Omega_1= H_1\setminus \left(\bigcup_{L \times H_2 \in \cal M_1}L\right), \quad
 \Omega_2= H_2\setminus \left(\bigcup_{H_1 \times L \in \cal M_2}L\right) $$
If $\Omega_1=\emptyset,$ then
$G=H_1\times H_2= \bigcup_{L \times H_2 \in \cal M_1}L\times H_2$, hence $\cal M=\cal M_1.$
In the same way, if $\Omega_2=\emptyset,$ then $\cal M=\cal M_2.$

So we may assume $\Omega_1\times \Omega_2\neq \emptyset.$
For $i\in \{1,2\},$ let $K_i$ be the intersection of the maximal normal subgroups  of $H_i.$ Notice that $H_i/K_i$ is isomorphic to a direct product of  simple groups and  $K_i$ is the smallest
subgroup of $H_i$ with this property. To fix our notation assume
$$H_1/K_1=\prod_{1\leq a \leq \alpha}S_a,\quad \quad H_2/K_2=\prod_{1\leq b \leq \beta}T_b$$
with $S_a$, $T_b$ simple groups.
To each $a\in A=\{1,\dots,\alpha\}$ there corresponds  the projection $\pi_{1,a}: H_1 \to S_a$ and to each $b \in B=\{1,\dots,\beta\}$
there corresponds  the projection $\pi_{2,b}: H_2 \to T_b.$
For $i\in\{1,2\},$ consider the projection $\rho_i: H_i\to H_i/K_i$ and the
image  $$\Delta_i=\{\rho_i(\omega)\mid\omega \in \Omega_i\}$$ of $\Omega_i$ under this projection.

By Lemma \ref{suz},
 to any $M\in \cal M_3$ we may associate a triple $(a,b,\phi)$ with $a\in A,$ $b\in B$ and $\phi: S_a\to T_b$
a group isomorphism such that
$$M=M(a,b,\phi)=\{(h_1,h_2)\in H_1\times H_2\mid \phi(\pi_{1,a}(h_1))=\pi_{2,b}(h_2)\}.$$
Now let $\Lambda$ be the set of the triples $(a,b,\phi)$  such that
$M(a,b,\phi)\in \cal M_3.$ By hypothesis, $\cal M_3$ contains no subgroup of index a prime number; this implies that
if $(a,b,\phi)\in \Lambda,$ then $S_a\cong T_b$ is a nonabelian simple group.

Now fix an element $(s_1,\dots,s_\alpha)\in \Delta_1$ and an element $x \in \Omega_1$ with $\rho_1(x)=(s_1,\dots,s_\alpha)$
and for each $(a,b,\phi)\in \Lambda$ let
$$U(a,b,\phi)=\{h\in H_2\mid \pi_{2,b}(h)\in \langle\phi(s_a)\rangle\}.$$
Clearly, since $T_b$ is a nonabelian simple group, $\langle\phi(s_a)\rangle\neq T_b$ and $U(a,b,\phi)$ is a proper subgroup of $H_2$.
Consider the following family of subgroups of $H_2:$
$$\cal T=\{M \mid H_1\times M \in \cal M_2\} \cup \{U(a,b,\phi)\mid (a,b,\phi) \in \Lambda\}.$$
We claim that $\cal T$ is a cover of $H_2.$ We have to prove that if $h_2 \in \Omega_2$,
then $h_2\in U(a,b,\phi)$ for some $(a,b,\phi)\in \Lambda.$
Observe that the elements of the set $\Om_1 \times \Om_2$ do not belong to any of the subgroups in  $\cal M_1$ or $\cal M_2$, thus the set   $\Om_1 \times \Om_2$ has to be covered by the subgroups in $\cal M_3$. In particular if $h_2 \in \Omega_2$, then $(x,h_2)\in M(a,b,\phi)$ for
some $(a,b,\phi)\in \Lambda.$ This implies that
$\pi_{2,b}(h_2)=\phi(\pi_{1,a}(x))=\phi(s_a)\in \langle \phi(s_a)\rangle$, hence $h_2\in U(a,b,\phi)$ and the claim is proved.
But this implies $|\cal M_1|+|\cal M_2|+|\cal M_3|=\sigma(G)\leq \sigma(H_2)\leq |\cal T|\leq |\cal M_2|+|\cal M_3|$
and, consequently, $\cal M_1=\emptyset.$ With a similar argument we deduce $\cal M_2=\emptyset.$
So if $\cal M_3\neq \emptyset,$ then $\cal M=\cal M_3$. By \cite[Lemma 1]{cohn}, there exists $M\in \cal M_3$ with $\sigma(G)\geq |G:M|+1$;
however $|G:M|=|S|$ for some nonabelian simple group $S$ which is an epimorphic image of $G$. This implies $\sigma(G)\leq \sigma(S)\leq |S|=|G:M|\leq \sigma(G)-1,$ a contradiction.

Let $\overline H_1=H_1\times 1$ and $\overline H_2=1\times H_2.$ We have proved that there exists $j\in \{1,2\}$, such that $\overline H_j\leq \bigcap_{1\leq i \leq \sigma}M_i.$ In particular $\overline H_jX_i\leq M_i$
for each $i\in\{1,\dots,\sigma\}$ hence, by Lemma \ref{effe}, we can conclude $\overline H_j\leq \bigcap_{1\leq i \leq \sigma}X_i.$
\end{proof}

\end{document}